\documentclass[11pt]{amsart}
\usepackage[utf8]{inputenc}
\usepackage[T1]{fontenc}
\usepackage{graphicx}
\usepackage{longtable}
\usepackage{wrapfig}
\usepackage{rotating}
\usepackage[normalem]{ulem}
\usepackage{amsmath}
\usepackage{amssymb}
\usepackage{capt-of}
\usepackage{hyperref}
\author{Thomas Hüttemann}\address[T.~Hüttemann]{Queen's University Belfast, School of Mathematics and Physics, Mathematical Sciences Research Centre, Belfast BT7~1NN, UK}\email{t.huettemann@qub.ac.uk}
\author{Dan Kucerovsky}\address[D.~Kucerovsky]{University of New Brunswick, Department of Mathematics and Statistics, P.O.~Box 4400, Fredericton, NB, Canada, E3B~5A3}\email{dkucerov@unb.ca}\thanks{Work on this paper started during a research visit of the second-named author to Queen's University Belfast. Financial support of the Fredrik and Catherine Eaton Fund is gratefully acknowledged}
\usepackage{amsthm}
\usepackage{amssymb}
\usepackage{tikz-cd}
\usepackage{mathtools}
\newtheorem{lem}{Lemma}
\newtheorem{prop}[lem]{Proposition}
\newtheorem{thm}[lem]{Theorem}
\newtheorem{cor}[lem]{Corollary}
\theoremstyle{definition}
\newtheorem{dfn}[lem]{Definition}
\newtheorem{xmpl}[lem]{Example}
\newtheorem{rem}[lem]{Remark}
\numberwithin{lem}{section}
\numberwithin{equation}{lem}
\newcommand{\bN}{\mathbb{N}}
\newcommand{\bZ}{\mathbb{Z}}
\newcommand{\bQ}{\mathbb{Q}}
\newcommand{\bR}{\mathbb{R}}

\newcommand{\rcof}[1][]{\arrow[r, tail,   "#1"']}
\newcommand{\rcofp}[1][]{\arrow[r, tail,   "#1"]}
\newcommand{\rweq}[1][]{\arrow[r, "\sim", "#1"']}
\newcommand{\po}{\ar[dr,phantom,"{\tikz{\draw[-] (-0.3,0) -- (-0.3,0.3) -- (0,0.3); \path (-.15,.15) node[circle,fill=black,scale=0.2] {};}}"]}
\newcommand{\C}{\mathcal{C}_{\text{\upshape H}}(R)}
\newcommand{\Cp}{\mathcal{C}_{\text{pre}}(R)}
\usepackage{enumerate}
\date{\today}
\title{An Euler characteristic for unbounded chain complexes}
\hypersetup{
 pdfauthor={Thomas Hüttemann and Dan Kucerovsky},
 pdftitle={An Euler characteristic for unbounded chain complexes},
 pdfkeywords={},
 pdfsubject={},
 pdfcreator={Emacs 29.3 (Org mode 9.6.15)}, 
 pdflang={English}}

\begin{document}

\begin{abstract}
\noindent We propose a definition of an Euler characteristic for
unbounded chain complexes by taking the (usual) Euler characteristics
of successively longer parts of the complex, weighted inversely
proportional to the length, and passing to the limit. This amounts to
taking the limit of the sequence of ranks of homology modules with
alternating signs in the sense of the Hölder summation method. We
establish the structure of a category with cofibrations and weak
equivalences on unbounded complexes for which the infinite Euler
characteristic is defined, and show that its Grothendieck
group \(K_0\) is unusually large (/viz./, uncountable).
\end{abstract}

\maketitle
\allowdisplaybreaks \vspace{-1cm}
\newcommand{\cH}{\chi_{\text{\upshape H}}}

\tableofcontents

\section*{Introduction}
\label{sec:orgf062321}
The Euler characteristic of a bounded chain complex~\(C\) of abelian
groups with finitely generated homology groups is defined to be the
alternating sum
\begin{equation*}
  \chi(C) = \sum_{k=0}^n (-1)^k \mathrm{rank}\,
  H_k(C) \tag{\(\mathcal{E}\)}
\end{equation*} 
of the ranks of its homology groups, assuming the complex is
concentrated in chain levels 0 to~\(n\). It is clear that the
definition extends immediately to complexes that are homologically
bounded (rather than bounded), and since \(\chi\) is homotopy
invariant, one can extend the definition further to complexes that are
homotopy equivalent to a complex of the type above.

However, none of these definitions capture properties of properly
unbounded complexes which arise, for example, in topology as singular
chain complexes of non-compact spaces.

In this paper, we develop aspects of a new Euler characteristic,
called the "infinite Euler characteristic", of certain unbounded
complexes. The underlying idea is to re-scale the
expression~(\(\mathcal{E}\)) to incorporate information on the
"length"~\(n\) of the sum, effectively averaging the ranks, and then
taking the limit \(n \to \infty\). In other words, we consider the
"Cesàro limit" of the sequence \(\left( (-1)^k \mathrm{rank}\, H_k(C)
\right)_{k \geq 0}\). This already enlarges the class of allowable
chain complexes dramatically.

One can go on step further: Instead of using the Cesàro limit as
described, one can iterate the construction and repeat the averaging
process several times. This procedure has been used by Hölder
\cite{zbMATH02705122} and is known as the Hölder summability method.

We introduce a category of unbounded chain complexes for which the
infinite Euler characteristic is defined, and equip it with the
structure of a category with cofibrations and weak equivalences so
that its higher algebraic \(K\)-groups are defined. Our main result is
that the Grothendieck group~\(K_{0}\) is rather large, and in fact
uncountable. This is highly significant as uncountable groups can
harbour interesting geometrical and topological properties. The proof
makes use of the theory of the infinite Euler characteristic developed
in the early parts of this paper.

\subsection*{Related work}
\label{sec:org89e5121}
Large Grothendieck groups are not a new phenomenom. In 2004, Gubeladze
\cite{zbMATH02105767} showed that for each field \(K\) of
characteristic~0 having infinite dimension as a \(\bQ\)-vector space
there exist (simplicial, quasi-projective) toric varieties over~\(K\)
with Grothendieck group of algebraic vector bundles having rank at
leat \(\dim_{\bQ} K\). This is achieved by gathering information about
various \(K_1\)-groups, and using an exact sequence argument to
establish the claim for~\(K_{0}\). In contrast, the approach taken in
the present paper is somewhat more direct, based on the \emph{explicit}
construction of a surjection from the Grothendieck group to~\(\bR\).

Euler characteristics have been studied for various mathematical
structures, sometimes involving divergent series; see, for example,
\cite{zbMATH05229957} for the case of finite categories, and the
references therein. The present paper differs in two respects: The
Euler characteristic is defined as a generalised limit of the
\emph{sequence} of the ranks of the homology groups with alternating signs,
and the main focus lies on embedding the Euler characteristic into the
framework for higher algebraic \(K\)-theory of chain complexes.

\subsection*{Structure of the paper}
\label{sec:orga76f37d}
We begin the paper by introducing the Hölder limiting process
in~\S\ref{sec:Hölder limits}. In \S\ref{sec:characteristic} we define the new
Euler characteristic~\(\cH\) and prove that its range is the set of
real numbers. Relevant examples of unbounded chain complexes arise
from topology, as is outlined in \S\ref{sec:topology}. We proceed to embed
the material into the wider context of higher algebraic \(K\)-theory.
We establish additivity of~\(\cH\) in \S\ref{sec:additivity}, introduce
the structure of a category with cofibrations and weak equivalences on
a category (unbounded) chain complexes in \S\ref{sec:Waldhausen}, and show
that its Grothendieck group~\(K_{0}\) is uncountable. Next, we show in
\S\ref{sec:finite_model} that there is a smaller model for the resulting
\(K\)-theory using chain complexes of finitely generated modules, and
finally discuss the possibility of computing the value of the new
Euler characteristic directly from the chain modules in \S\ref{sec:compute}.

\section{Hölder limits}
\label{sec:orga63ca80}
\label{sec:Hölder limits} We denote the set of positive integers \(\{1,\,
2,\, 3,\, \cdots\}\) by~\(\bN\).  The \emph{Cesàro matrix} is the infinite
\(\bN\)-indexed square matrix defined by
\begin{equation*}
  M = (m_{i,j}) =
  \begin{pmatrix}
    1 & 0 & 0 & 0 & 0 & \ldots \\[0.3em]
    \frac12 & \frac12 & 0 & 0 & 0 & \ldots \\[0.3em]
    \frac13 & \frac13 & \frac13 & 0 & 0 & \ldots \\[0.3em]
    \frac14 & \frac14 & \frac14 & \frac14 & 0 & \ldots \\[0.3em]
    \vdots & \vdots & \vdots & \vdots & \ddots & \ddots
  \end{pmatrix}
  \ ,
\end{equation*}
that is, \(m_{i,j} = 1/i\) if \(i \geq j\), and \(m_{i,j} = 0\)
otherwise.

Let \(\mathbf{a} = (a_n)_{n \in \mathbb{N}} \in \mathbb{R}^\infty\) be
a sequence of real numbers. The sequence \(\mathbf{a}\) is
\emph{\(M^k\)-convergent with limit \(\ell \in \mathbb{R}\)} if the
sequence \(M^k \mathbf{a}\) converges to~\(\ell\). It is well known,
and an exercise in elementary analysis, that if \(\mathbf{a}\) is
\(M^k\)-convergent, then it is also \(M^m\)-convergent for every \(m
\geq k\), with the same limit. Thus it makes sense to define
\(\mathbf{a}\) to be \emph{Hölder convergent} or H-\emph{convergent} if the
sequence is \(M^k\)-convergent for some \(k > 0\), and define the
Hölder limit \(\mathrm{\relax H}\hbox{-}\lim(\mathbf{a})\) by \(\lim
M^k \mathbf{a}\) for sufficiently large \(k \gg 0\). If \(\mathbf{a}\)
is H-convergent with limit~0 we say that \(\mathbf{a}\) is an H-\emph{null}
sequence.

\begin{xmpl}
Suppose \(\mathbf{a}\) is the divergent sequence given by \(a_n =
(-1)^{n-1} \cdot \lceil \frac{n}{2} \rceil\), that is, the sequence
\(1,\,-1,\, 2,\,-2,\,3,\,-3,\, \cdots\). Then \(M\mathbf{a}\) is the
sequence \(1,\, 0,\, \frac23,\, 0,\, \frac35,\, 0,\, \frac47,\, 0,\,
\cdots\), that is, \((M\mathbf{a})_{2n-1} = n/(2n-1)\) and
\((M\mathbf{a})_{2n} = 0\). Thus \(M\mathbf{a}\) is divergent, with
accumulation points \(0\) and~\(1/2\). The sequence \(M^2\mathbf{a}\)
is convergent with limit~\(1/4\), which is the H-limit
of~\(\mathbf{a}\).
\end{xmpl}

Given any sequence \(\mathbf{a} = (a_n)\) we will write
\(\frac1n \mathbf{a}\) or \(\mathbf{a}/n\) for the sequence
\((a_n/n)_{n \in \bN}\). More generally, if \(f \colon \bN \to \bR\),
\(n \mapsto f(n)\) is a function we write \(\mathbf{a}/f\) for the
sequence \((a_n/f(n))_{n \in \bN}\).

\begin{dfn}
We write \(\mathbf{a} = o(f)\) if the sequence \(\mathbf{a}/f =
(a_n/f(n))_{n \in \bN}\) is a null sequence (in the usual sense), and
\(\mathbf{a} = \mathrm{\relax H}\hbox{-}o(f)\) if \(\mathbf{a}/f\) is
an H-null sequence.
\end{dfn}

\begin{lem}
\label{lem:H-conv-H-on}
If \(\mathbf{a}\) is an \(M^k\)-convergent sequence then \(a_n =
o(n^k)\), that is, \(\lim_{n \to \infty} a_n/n^k = 0\).
\end{lem}

\begin{proof}
We use induction on~\(k\). For \(k = 1\) we compute
\((M\mathbf{a})_n - (M\mathbf{a})_{n-1} = a_n/n -
(M\mathbf{a})_{n-1}/n\).  As \(M\mathbf{a}\) converges in this case,
\((M\mathbf{a})_n - (M\mathbf{a})_{n-1} \to 0\) and
\((M\mathbf{a})_{n-1}/n \to 0\) as \(n \to \infty\), thus \(a_n/n \to
0\) whence \(a_n = o(n)\).

For \(k > 1\) we observe \(M^k\mathbf{a} = M^{k-1} (M\mathbf{a})\) by
associativity. Thus if \(\mathbf{a}\) is \(M^k\)-convergent then
\(M\mathbf{a}\) is \(M^{k-1}\)-convergent, and by induction hypothesis
\((M\mathbf{a})_n = o(n^{k-1})\). We compute
\begin{align*}
  \frac{1}{n^{k-1}} (M\mathbf{a})_n - \frac{1}{(n-1)^{k-1}}
  (M\mathbf{a})_{n-1} \hspace{-4em} & \\
  & = \frac{1}{n^{k-1}} \sum_{1}^{n} \frac{a_j}{n}
  - \frac{1}{(n-1)^{k-1}} \sum_{1}^{n-1} \frac{a_j}{n-1} \\
  & = \frac{1}{n^k} \sum_1^n a_j - \frac{1}{(n-1)^k} \sum_1^{n-1} a_j \\
  & = \frac{a_n}{n^k} + \frac{(n-1)^k - n^k}{n^{k}(n-1)^k} \cdot
  \sum_1^{n-1} a_j \\
  & = \frac{a_n}{n^k} + \frac{q}{n^k} \cdot
  \frac{1}{(n-1)^{k-1}} (M\mathbf{a})_{n-1} \ ,
\end{align*}
where \(q = (n-1)^k - n^k\) is a polynomial of degree~\(k-1\)
in~\(n\). Since \((M\mathbf{a})_n = o(n^{k-1})\) by induction, the
very last summand vanishes in the limit \(n \to \infty\), and the
left-hand side also vanishes in the limit. Thus also \(a_n/n^k \to 0\)
whence \(a_n = o(n^k)\) as claimed.
\end{proof}

\begin{dfn}
The sequence \(\mathbf{a} = (a_n)_{n \in \bN}\) is called \emph{absolutely
H-convergent} if the sequence \(|\mathbf{a}| = (|a_n|)_{n \in \bN}\)
is H-convergent. We say that \(\mathbf{a}\) is \emph{absolutely H-null} if
\(\mathbf{a}\) is absolutely H-convergent with H-limit~0.
\end{dfn}

\begin{lem}
\label{H-o-n-order} Suppose that \(\mathbf{a}\) and \(\mathbf{b}\) are
sequences of non-negative numbers satisfying \(\mathbf{b} \leq
\mathbf{a}\) (termwise), and that \(f \colon \bN \to \bR\) is a
function taking positive values only. If \(\mathbf{a} = \mathrm{\relax
H}\hbox{-}o(f)\) then also \(\mathbf{b} = \mathrm{\relax
H}\hbox{-}o(f)\).
\end{lem}

\begin{proof}
The inequalities \(0 \leq b_n \leq a_n\), valid for all~\(n\), imply
the inequalities \(0 \leq \big(M^k (\mathbf{b}/f)\big)_n \leq \big(M^k
(\mathbf{a}/f)\big)_n\). Thus if \(\big(M^k (\mathbf{a}/f)\big)_n \to
0\) then also \(\big(M^k (\mathbf{b}/f)\big)_n \to 0\).
\end{proof}

\begin{lem}
\label{H-null-implies-H-null}
If \(\mathbf{a}\) is absolutely H-null, then \(\mathbf{a}\) is H-null.
\end{lem}

\begin{proof}
It suffices to show that, given \(k \gg 0\), if \(M^k |\mathbf{a}|\)
is a null sequence then \(M^k \mathbf{a}\) is a null sequence (in the
sense of usual limits). Write \(M^k_{n,j}\) for the \((n,j)\)-entry of
the matrix~\(M^k\). Then since \(M^k_{n,j} \geq 0\),
\begin{equation*}
  \Big| (M^k \mathbf{a})_n \Big| =
  \bigg| \sum_{j=1}^n M^k_{n,j} a_j \bigg| \leq
  \sum_{j=1}^n M^k_{n,j} |a_j| = (M^k |\mathbf{a}|)_n \ ,
\end{equation*}
and the right-hand side vanishes in the limit \(n \to \infty\) by
hypothesis.
\end{proof}

Applying this to the sequences \(\mathbf{a}/n\) and \(|\mathbf{a}/n|
= |\mathbf{a}|/n\) immediately gives:

\begin{cor}
\label{H-on-implies-H-on}
If \(|\mathbf{a}| = \mathrm{\relax H}\hbox{-}o(n)\) then \(\mathbf{a}
= \mathrm{\relax H}\hbox{-}o(n)\). \qed
\end{cor}

Given a sequence \(\mathbf{a} = (a_k)_{k \in\bN}\) we let
\(\mathbf{a}[m] = (b_k)_{k \in \bN}\) denote its \emph{\(m\)-th shift},
where
\begin{equation*}
  b_k =
  \begin{cases}
    a_{k-m} & \text{if } k-m > 0, \\
    0 & \text{else}.
  \end{cases}
\end{equation*}

\begin{lem}
\label{lem:shift_invariance} Let \(m\) be an integer. The sequence
\(\mathbf{a}\) is H-convergent if and only if the sequence
\(\mathbf{a}[m]\) is H-convergent, in which case their H-limits
agree. In particular, if \(\mathbf{a}\) has H-limit~\(s\) and
\(\mathbf{b}\) is obtained from~\(\mathbf{a}\) by omitting, inserting
or modifying finitely many terms, then \(\mathbf{b}\) also has
H-limit~\(s\).
\end{lem}

For a proof see p419, \S100 of~\cite{zbMATH03127712}.

\begin{rem}
By its very definition, the notion of \(M^k\)-convergence is a "matrix
method" where a sequence is transformed by application of a single
matrix, possibly converting a non-convergent sequence into a
convergent one. It can be shown that, in contrast, H-convergence is
\emph{not} a matrix method, \emph{cf.}~\cite[Satz~8.1]{zbMATH03070559}.
\end{rem}

\section{Admissible chain complexes and the infinite Euler characteristic}
\label{sec:org8bbf230}
\label{sec:characteristic}

For the remainder of the paper we will work with modules over an
arbitrary (unital) commutative Noetherian integral domain~\(R\); the
reader may want to keep the specific example \(R = \bZ\) in mind
(where \(R\)-modules are nothing but abelian groups) which will
illustrate the salient points of the theory in full.

\begin{dfn}
The \emph{rank} of an \(R\)-module \(A\) is defined as
\begin{equation*}
  \mathrm{rank}\,A = \dim_{Q} (A \otimes_{R} Q)
\end{equation*}
where \(Q\) is the field of fractions of~\(R\), and "\(\dim\)" denotes
the vector space dimension.
\end{dfn}

In the motivating case \(R = \bZ\) we have \(Q = \bQ\). --- In
general, the field \(Q\) is a localisation of the ring~\(R\), hence a
flat \(R\)-module. This means that from any short exact sequence of
\(R\)-modules
\begin{equation}
\label{eq:ses-modules}
0 \to A_1 \to A_2 \to A_3 \to 0
\end{equation}
we obtain a short exact sequence of \(Q\)-vector space
\begin{equation*}
  0 \to A_1 \otimes_R Q \to A_2 \otimes_R Q \to A_3 \otimes_R Q \to 0
\end{equation*}
which implies:

\begin{lem}
\label{lem:rank_additive} The rank is an additive function. More
precisely, for every short exact sequence~\ref{eq:ses-modules} we obtain the
additivity relation
\begin{equation*}
  \mathrm{rank}\,(A_1) - \mathrm{rank}\,(A_2) + \mathrm{rank}\,(A_3) =
  0 \ . \tag*{\qedsymbol}
\end{equation*}
\end{lem}

\begin{dfn}
\label{def:pre-adm} Let \(C\) be a chain complex of \(R\)-modules. We call
\(C\) a \emph{Hölder pre-admissible chain complex}, or just
\emph{pre-admissible}, if

\begin{enumerate}[{\rm (a)}]
\item \(C\) vanishes in negative chain degrees (so \(C\) is a "positive"
chain complex), that is, \(C_{n} = 0\) for \(n < 0\),
\item all homology modules \(H_n C\) are finitely generated \(R\)-modules,
\item \label{main_cond} the sequence
\begin{equation*}
  \mathbf{H}C = \big( (-1)^{n-1} \cdot \mathrm{rank}\,
  H_{n-1} C\big)_{n \in \mathbb{N}}
\end{equation*}
is H-convergent.
\end{enumerate}

We denote the category of pre-admissible chain complexes and arbitrary
chain morphisms by the symbol~\(\Cp\).
\end{dfn}

Note that by Lemma~\ref{lem:H-conv-H-on} the sequence \(\mathbf{H}C/n^k\)
associated with a pre-admissible complex~\(C\) converges to~0 for \(k
\gg 0\). We will later want to impose the condition that the sequence
\(\mathbf{H}C/n\) be \emph{absolutely} H-null:

\begin{dfn}
\label{def:adm} A Hölder pre-admissible chain complex \(C\) is called
\emph{Hölder admissible}, or just \emph{admissible}, if it satisfies
\(|\mathbf{H}C| = \mathrm{\relax H} \hbox{-}o(n)\). Explicitly, this
means that the sequence
\begin{equation*}
  \Big(\frac 1n \mathrm{rank}\,H_{n-1}C \Big)_{n \in \bN}
\end{equation*}
is H-convergent with H-limit~0.

We denote the full subcategory of~\(\Cp{}\) consisting of the
admissible chain complexes and arbitrary chain morphisms by the
symbol~\(\C{}\).
\end{dfn}

\begin{dfn}
The \emph{Hölder method infinite Euler characteristic}, or \emph{HE
  characteristic} for short, of the pre-admissible chain complex \(C
  \in \Cp{}\) is defined by
\begin{equation*}
  \cH(C) = \mathrm{\relax H}\hbox{-}\lim \mathbf{H}C \ .
\end{equation*}
\end{dfn}

By construction, \(\cH\) is invariant under quasi-isomorphism of chain
complexes. --- By Lemma~\ref{lem:shift_invariance}, \(\cH(C) = \cH(D)\)
whenever the sequence \(\mathbf{H}D\) can be obtained
from~\(\mathbf{H}C\) by inserting, deleting or modifying finitely many
entries, paying attention to the signs in the definition
of~\(\mathbf{H}C\). For example, we have the following result:

\begin{prop}
\label{prop:minus} For any \(m \in \mathbb{Z}\), we have \(\mathbf{H}
\big(C[m]\big) = (-1)^m (\mathbf{H}C)[m]\) and hence \(\cH
\big(C[m]\big) = (-1)^m \cH(C)\). Here \(C[m]\) is the \(m\)th
shift suspension of~\(C\), with \(C[m]_k = C_{k-m}\). \qed
\end{prop}

We omit the easy proof of the following Proposition; in any case, a
stronger additivity property will be discussed and proved in the next
section.

\begin{prop}
\label{prop:plus} The function \(\cH\) is additive with respect to
direct sums: For \(C,\, D \in \Cp{}\) the equality \(\cH(C
\oplus D) = \cH(C) + \cH(D)\) holds. \qed
\end{prop}

The function \(\cH\) is not interesting for bounded complexes. In
fact, one easily verifies:

\begin{prop}
\label{prop:bounded} A homologically bounded chain complex~\(C\) has HE
characteristic \(\cH(C) = 0\), since \(\mathrm{\relax H}\hbox{-}\lim
\mathbf{H}C = \lim \mathbf{H}C = 0\). In particular, the HE
characteristic of a complex vanishes whenever the classical Euler
characteristic is defined. \qed
\end{prop}

\begin{xmpl}
  Let \(C\) be the \(\bZ\)-module chain complex defined by
  \(C_{2n} = \mathbb{Z}\) and \(C_{2n+1} = 0\). Then
  \(H_{2n}C = \mathbb{Z}\) and \(H_{2n+1}C = 0\), and the sequence
  \(\mathbf{H}C = (1,\, 0,\, 1,\, 0,\, 1,\, \cdots)\) does not
  converge. In contrast,the sequence
  \(M \cdot \mathbf{H}C = (1,\, \frac12,\, \frac23,\, \frac12,\,
  \frac35,\, \frac12,\, \frac47,\, \cdots)\) converges to~1/2 so that
  \(\cH(C) = \mathrm{\relax H}\hbox{-}\lim \mathbf{H}C = 1/2\).
\end{xmpl}

\begin{xmpl}
\label{1 over m} Generalising the previous example, fix \(m \geq 1\) and
let \(C\) be the chain complex defined by \(C_n = \bZ\) if \(2m \mid
n\), and \(C_n = 0\) otherwise. Then writing \(n = 2mq + k\) with \(0
\leq k < 2m\), the \(n\)th term of \(M \cdot \mathbf{H}C\) is
\((q+1)/(2mq+k)\). Thus taking the limit \(n \to \infty\), that is,
\(q \to \infty\), we see \(\cH(C) = \mathrm{\relax H}\hbox{-}\lim
\mathbf{H}C = 1/2m\).
\end{xmpl}

\begin{prop}
For every rational number~\(q \in \bQ\) there exists an admissible
chain complex \(D \in \C{}\) with \(\cH(D) = q\).
\end{prop}

\begin{proof}
Of course \(\cH(0)=0\). If \(q > 0\) write \(q = 2a/2b\) with \(a,b
\in \bN\), choose a complex~\(C \in \C{}\) with \(\cH(C) = 1/2b\)
according to Example~\ref{1 over m}, and let \(D = \bigoplus_{2a} C\). Then
\(\cH(D) = 2a \cdot \cH(C) = 2a/2b = q\) by
Proposition~\ref{prop:plus}. Since \(|\mathbf{H}C|_n \leq 2b\) by
construction, \(\mathbf{H}C = \mathrm{\relax H}\hbox{-}o(n)\) so \(C\)
is an admissible chain complex. --- If \(q<0\), apply this
construction to~\(-q\) and note that then \(\cH \big( C[1] \big) = q\)
by Proposition~\ref{prop:minus}.
\end{proof}

Even better (though less explicit in construction):

\begin{thm}
\label{thm:all_R} For every real number~\(r \in \bR\) there exists an
admissible chain complex \(C \in \C{}\) with \(\cH(C) = r\).
\end{thm}

\begin{proof}
By Proposition~\ref{prop:minus} it is enough to deal with the case \(r >
0\). Choose non-negative integers \(a\) and \(b\) such that \(a < r <
b\). We will define a chain complex~\(C\) with trivial (zero)
differentials, which implies that \(H_nC = C_n\). We start by
declaring \(C_{2n+1} = 0\) for all \(n \geq 0\). The even-numbered
chain modules are chosen as follows:
\begin{itemize}
\item We let \(C_0 = R^{2b}\) and note that \((M \mathbf{H}C)_1 = 2b >
  r\) and \((M \mathbf{H}C)_{2} = 2b/2 = b > r\).
\item We let \(C_j = R^{2a}\) for \(j = 2,\, 4,\, \cdots,\, 2m_1\),
where \(m_1 > 0\) is chosen minimal so that \((M \mathbf{H}C)_{2m_1}
  < r\).
\item We let \(C_j = R^{2b}\) for \(j = 2m_1+2,\, 2m_1+4,\, \cdots,\,
  2m_2\), where \(m_2 > m_1\) is chosen minimal so that \((M
  \mathbf{H}C)_{2m_2} > r\).
\item We let \(C_j = R^{2b}\) for \(j = 2m_2+2,\, 2m_1+4,\, \cdots,\,
  2m_3\), where \(m_3 > m_2\) is chosen minimal so that \((M
  \mathbf{H}C)_{2m_3} < r\).
\item And so on.
\end{itemize}
The resulting sequence \(M \mathbf{H}C\) oscillates around the
value~\(r\), by construction, with amplitude of oscillation bounded by
\((b-a)/n\) at the \(n\)th term so that \(\cH(C) = \lim M \mathbf{H}C
= r\). It is clear that \(|\mathbf{H}C| = o(n)\) by construction since
\(\mathrm{rank}\, H_{n}C\) takes one of the three values 0, \(2a\)
and~\(2b\). Thus in particular \(|\mathbf{H}C| = \mathrm{\relax
H}\hbox{-}o(n)\), making \(C\) a Hölder admissible complex.
\end{proof}

Except for the trivial case \(q=a=b=0\), the chain complex produced in
this proof will have infinitely many non-vanishing homology modules so
that its classical Euler characteristic is \emph{not} defined.

\section{Examples from topology}
\label{sec:orgfa07626}
\label{sec:topology}

Admissible chain complexes arise from suitable topological spaces, for
example by taking the singular chain complex \(C = C(X,R)\) with
coefficients in the ring~\(R\). If \(X\) is a CW-complex we can
consider the cellular chain complex with \(n\)th chain module
a free \(R\)-module with basis the \(n\)-cells of~\(X\). It gives the same
homology modules as the singular complex, and is often easier to use
for computations.

\begin{dfn}
\label{def:Euler_char_X}
We define the \emph{infinite Euler characteristic of the topological
space~\(X\)} by setting \(\cH(X;R) = \cH(C)\), provided the
chain complex \(C\) is pre-admissible.
\end{dfn}

The following facts are immediate from the previous section:

\begin{prop}
\label{prop:topology}
\begin{enumerate}[{\rm (a)}]
\item \label{finitely_many} If \(X\) has only finitely many non-vanishing
homology modules, then \(\cH(X;R) = 0\).
\item \label{homotopy_equivalent} If \(X\) and~\(Y\) are homotopy equivalent,
then \(\cH(X;R) = \cH(Y;R)\).
\end{enumerate}
\end{prop}

\begin{xmpl}
If \(X\) is a finite or, more generally, finite-dimensional
CW-complex, then \(\cH(X;R) = 0\). This follows from
Proposition~\ref{prop:topology}~(\ref{finitely_many}) upon inspection of the
cellular chain complex of~\(X\).
\end{xmpl}

\begin{xmpl}
Proposition~\ref{prop:topology} implies that \emph{every contractible space~\(X\)
has vanishing infinite Euler characteristic}. This applies to the
following spaces:

\begin{enumerate}
\item The infinite-dimensional sphere \(S^{\infty} = \bigcup_{k \geq 0}
   S^k\), the union taken with respect to the embedding of \(S^k\)
into~\(S^{k+1}\) as the "equator" and equipped with the weak
topology such that \(F \subseteq S^{\infty}\) is closed if and only
if \(F \cap S^{k} \subseteq S^{k}\) is closed for all \(k \geq 0\).
\item The general linear and unitary groups of an infinite dimensional
separable (real, complex or quaternionic) Hilbert space, equipped
with the norm topology \cite[Theorems 2 and 3]{zbMATH03211541}.
\item The unitary group in \(\mathcal{M}(A \otimes \mathbb{K})\), where
\(A\) is a \(\sigma\)-unital C\textsuperscript{*}-algebra and \(\mathbb{K}\)
denotes the set of compact operators of the Hilbert
space~\(\ell^2\), see \cite{zbMATH03998548} and
\cite[Theorem~16.8]{zbMATH00412105}.
\end{enumerate}
\end{xmpl}

\begin{xmpl}
The infinite wedge of spheres \(X = \bigvee_{k \geq 1} S^{k}\) has
homology \(H_k (X; R) \cong R\) so that the singular chain
complex \(C\) of~\(X\) with coefficients in~\(R\) satisfies
\((\mathbf{H}C)_k = (-1)^{k-1}\) for all \(k \in \bN\). Consequently,
\(\cH(X;R) = 0\).
\end{xmpl}

\begin{xmpl}
  The infinite one-point union of odd-dimensional spheres
  \(Y = \bigvee_{k \geq 0} S^{2k+1}\) has homology
\begin{equation*}
  H_k (Y; R) \cong
  \begin{cases}
    R & \text{for \(k\) odd,} \\
    R & \text{for \(k=0\),} \\
    0 & \text{else}
  \end{cases}
\end{equation*}
so that the singular chain complex~\(C\) of~\(Y\) satisfies
\((\mathbf{H}C)_{2k} = -1\) and \((\mathbf{H}C)_{2k+1} = 0\) for \(k
\in \bN\). Consequently, \(\cH(X;R) = -1/2\).
\end{xmpl}

\section{Additivity of the infinite Euler characteristic}
\label{sec:org6a797c6}
\label{sec:additivity}

A short exact sequence
\begin{equation}
\label{ses}
\mathcal{S} \colon \quad 0 \to A \xrightarrow{f} B \xrightarrow{g} C \to 0
\end{equation}
of pre-admissible chain complexes \(A,\, B,\, C \in \Cp{}\)
yields a long exact sequence of homology modules
\begin{equation}
\label{les}
\ldots \to H_{n+1} C \xrightarrow{\delta_{n+1}}
  H_n A \xrightarrow{H_n(f)} H_n B \xrightarrow{H_n(g)} H_n C
  \xrightarrow{\delta_n} H_{n-1} A \to \ldots 
\end{equation}
terminating with a surjection \(H_0 B \to H_0 C \to 0\). Truncating
to the left at stage~\(n\) yields a finite exact sequence
\begin{multline*}
  0 \to \mathrm{im}\, \delta_{n} \to H_{n-1}A \to H_{n-1}B \to H_{n-1}C \\
  \to H_{n-2} A \to H_{n-2} B \to H_{n-2} C \\
  \ddots \\
  \qquad \qquad \qquad \qquad \to H_1 A \to H_1 B \to H_1 C \\
  \to H_0 A \to H_0 B \to H_0 C \to 0
\end{multline*}
of finitely generated \(R\)-modules. By exactness, the alternating sum
of the ranks of the modules in this truncated sequence vanishes, that
is,
\begin{multline*}
  \sum_0^{n-1} (-1)^j \cdot \mathrm{rank}\, H_j A
  - \sum_0^{n-1} (-1)^j \cdot \mathrm{rank}\, H_j B
  + \sum_0^{n-1} (-1)^j \cdot \mathrm{rank}\, H_j C \\
  = (-1)^{n-1} \cdot \mathrm{rank}\, \mathrm{im}\, \delta_{n} \ .
\end{multline*}
Upon division by~\(n\) we obtain
\begin{multline*}
  \frac1{n} \sum_0^{n-1} (-1)^j \cdot \mathrm{rank}\, H_j A
  - \frac1{n} \sum_0^{n-1} (-1)^j \cdot \mathrm{rank}\, H_j B
  + \frac1{n} \sum_0^{n-1} (-1)^j \cdot \mathrm{rank}\, H_j C \\
  = \frac{1}{n} \cdot (-1)^{n-1} \mathrm{rank}\, \mathrm{im}\, \delta_{n} \ ,
\end{multline*}
that is, we have the relation
\begin{equation*}
(M\mathbf{H}A)_{n} - (M\mathbf{H}B)_{n} + (M\mathbf{H}C)_{n}
  = \frac{1}{n} \cdot (-1)^{n-1} \mathrm{rank}\, \mathrm{im}\, \delta_{n}
  \qquad \text{(\(n \geq 1\))}
\end{equation*}
or, writing \((\boldsymbol{\delta}\mathcal{S})_n = (-1)^{n-1}
\mathrm{rank}\, \mathrm{im}\, \delta_n\),
\begin{equation}
\label{add-eq}
M\mathbf{H}A - M\mathbf{H}B + M\mathbf{H}C
  = \boldsymbol{\delta}\mathcal{S}/n \ .
\end{equation}

\begin{dfn}
We say that the short exact sequence~\(\mathcal{S}\) of~(\ref{ses}) is
\emph{weakly admissible} if the sequence
\(\boldsymbol{\delta}\mathcal{S}/n\) is H-null, that is, if
\(\boldsymbol{\delta}\mathcal{S} = \mathrm{\relax H}\hbox{-}o(n)\). We
say that \(\mathcal{S}\) is \emph{admissible} if if the sequence
\((\mathrm{rank}\, \mathrm{im}\, \delta_{n}/n)_{n \in \bN}\) is
H-null, that is, if \(|\boldsymbol{\delta}\mathcal{S}| =
\mathrm{H}\hbox{-}o(n)\).
\end{dfn}

Every admissible sequence is weakly admissible by
Corollary~\ref{H-on-implies-H-on}. We observe that any split short exact
sequence \(\mathcal{A} \colon \ 0 \to A \to A \oplus C \to C \to 0\)
of pre-admissible chain complexes is admissible; indeed, the
connecting homomomorphisms~\(\delta_n\) are trivial in this case
(equivalently, the homomorphisms \(H_{n-1}(f)\) are injective so that
\(\mathrm{im}\, \delta_{n} = \ker H_{n-1}(f) = 0\)) whence
\(\boldsymbol{\delta}\mathcal{A}\) is the zero-sequence. ---
Admissibility of short exact sequences is a void condition as long as
we are only dealing with admissible chain complexes. In fact, we have
the following result:

\begin{lem}
\label{lem:admissible-admissible}
If one of \(A\) and \(C\) is admissible, then \(\mathcal{S}\) is
admissible.
\end{lem}

\begin{proof}
If \(C\) is admissible then \(|\mathbf{H}C| = \mathrm{\relax
H}\hbox{-}o(n)\). As \(\mathrm{im}\, \delta_n\) is a homomorphic image
of~\(H_n(C)\) we have the inequality
\begin{equation*}
  \relax | \boldsymbol{\delta}\mathcal{S}|_n
  = \mathrm{rank}\, \mathrm{im}\, \delta_n \leq \mathrm{rank}\, H_{n} (C)
  = |\mathbf{H}C|_{n+1} \ ,
\end{equation*}
whence \(|\boldsymbol{\delta}\mathcal{S}| = \mathrm{\relax
H}\hbox{-}o(n)\) by Lemma~\ref{H-null-implies-H-null}. If \(A\) is
admissible we similarly use that \(\mathrm{im}\, \delta_n\) is a
submodule of~\(H_{n-1}(A)\) so that \(\mathrm{rank}\, \mathrm{im}\,
\delta_n \leq \mathrm{rank}\, H_{n-1} (A) = |\mathbf{H} A|_{n}\)
implies \(|\boldsymbol{\delta}\mathcal{S}| = \mathrm{\relax
H}\hbox{-}o(n)\).
\end{proof}

\medbreak

If \(\mathcal{S}\) is weakly admissible, we can pass to the H-limit in
equation~(\ref{add-eq}) and note that the right-hand side vanishes, by the
very definition of weak admissibility. In the limit, the equation
turns into \(\cH(A) - \cH(B) + \cH(C) = 0\). We have proved:

\begin{thm}
\label{additive}
The infinite Euler characteristic \(\cH\) is additive on weakly
admissible short exact sequences of pre-admissible chain
complexes. That is, given the weakly admissible exact
sequence~\(\mathcal{S}\) of~(\ref{ses}), we have the relation \(\cH(A) -
\cH(B) + \cH(C) = 0\). \qed
\end{thm}

For later use, we record the following result:

\begin{lem}
\label{in_C} Let \(\mathcal{S} \colon \ 0 \to A \to B \to C \to 0\) be a
short exact sequence of chain complexes of \(R\)-modules.
\begin{enumerate}[{\rm (a)}]
\item \label{in_C_1} If two of the three complexes \(A\), \(B\) and~\(C\)
have finitely generated homology modules, then so does the third.
\item \label{in_C_2} Suppose that \(A\), \(B\) and~\(C\) vanish in negative
chain degrees.  If two of the chain complexes \(A\), \(B\)
and~\(C\) are admissible (that is, are objects of~\(\C{}\)),
then so is the third.
\end{enumerate}
\end{lem}

\begin{proof}
The first part follows immediately from the long exact homology
sequence
\begin{equation*}
  \ldots \to H_{n+1}(C) \to H_n(A) \to H_n(B) \to H_n(C) \to
  H_{n-1}(A) \to \ldots
\end{equation*}
as any \(R\)-modules "sandwiched" between two finitely generated
\(R\)-modules in a long exact sequence must be finitely generated
itself.

For the second part, note that Lemma~\ref{lem:admissible-admissible}
guarantees that \(|\boldsymbol{\delta}\mathcal{S}| = \mathrm{\relax
H}\hbox{-}o(n)\). By splitting the long exact sequence~\ref{les} we obtain
exact sequences
\begin{equation*}
  0 \to \mathrm{im}\, \delta_{n}
  \to H_{n-1}(A) \to H_{n-1}(B) \to H_{n-1}(C)
  \to \mathrm{im}\, \delta_{n-1} \to 0
\end{equation*}
(where \(\delta_0\) is the zero map) which implies
\begin{equation*}
  \big|\boldsymbol{\delta}\mathcal{S}[-1] \big| +
  \big|\boldsymbol{\delta}\mathcal{S}\big| =
  \big| \mathbf{H}A  \big| -
  \big| \mathbf{H}B  \big| +
  \big| \mathbf{H}C  \big| \ .
\end{equation*}
The left-hand side is \(\mathrm{\relax H}\hbox{-}o(n)\) by hypothesis
and Lemma~\ref{lem:shift_invariance}, and so are two of the summands on the
right-hand side. Hence the third summand must be \(\mathrm{\relax
H}\hbox{-}o(n)\) as well, as required.
\end{proof}

\begin{thm}
Suppose that \(X\) is a CW-complex, and that \(A \subseteq X\) is a
subcomplex of~\(X\). Suppose further that two of the three spaces
\(A\), \(X\) and~\(X/A\) are Hölder admissible in the sense that their
cellular chain complexes are Hölder admissible
(Definition~\ref{def:adm}). Then all three spaces are Hölder
admissible, and \(\cH(X/A; R) = \cH(X;R) - \cH(A;R)\).
\end{thm}

\begin{proof}
The (reduced) cellular chain complex of~\(X/A\) is isomorphic to the
quotient of the cellular chain complex of~\(X\) by its subcomplex
formed by the cellular chain complex of~\(A\). Admissibility thus
follows immediately from Lemma~\ref{in_C}~(\ref{in_C_2}), and the formula for the
infinite Euler characteristic is a consequence of Theorem~\ref{additive}.
\end{proof}

\section{Algebraic \(K\)-theory}
\label{sec:orgc82a9bc}
\label{sec:Waldhausen}

We equip the category~\(\C{}\) of admissible chain complexes
with the structure of a category with cofibrations and weak
equivalences in the sense of Waldhausen \cite{zbMATH03927168}, as follows:

\begin{itemize}
\item A \emph{cofibration} is an injective chain map \(f \colon C \to D\)
in~\(\C{}\). The subcategory of \(\C{}\) consisting of
all objects of~\(\C{}\) and all cofibrations will be denoted
\(\mathrm{co}\C{}\). As usual, cofibrations will be written
with the symbol "\(\tikzcd {} \rcof & {} \endtikzcd\)".
\item A \emph{weak equivalence} is a chain map in~\(\C{}\) that is a
quasi-iso\-mor\-phism (that is, a chain map inducing isomorphisms on
all homology modules). The subcategory of~\(\C{}\) consisting of all
objects of~\(\C{}\) and all weak equivalences will be denoted
\(\mathrm{w}\C{}\). As usual, weak equivalences will be written with
the symbol "\(\tikzcd {} \rweq & {} \endtikzcd\)".
\end{itemize}

\begin{thm}
\begin{enumerate}[{\rm (a)}]
\item \label{Waldhausen-structure} With the definitions above, \(\C{}\) is a
category with cofibrations and weak equivalences in the sense of
Waldhausen \cite{zbMATH03927168} satisfying the
saturation axiom and the extension axiom.
\item \label{cylinder} The usual mapping cylinder construction for chain
complexes provides a cylinder functor satisfying the cylinder
axiom.
\end{enumerate}
\end{thm}

\begin{proof}
(\ref{Waldhausen-structure}) Axioms Cof~1 (isomorphisms are cofibrations)
and Cof~2 (all objects are cofibrant) are trivial. Axiom Cof~3
requires that for each cofibration \(f \colon \tikzcd A \rcof & B
\endtikzcd\) and each map \(A \to C\) in~\(\C{}\), the pushout
\begin{equation*}
  \begin{tikzcd}
    A \rcofp[f] \dar{} \po & B \dar{} \\
    C \rar{F} & D
  \end{tikzcd}
\end{equation*}
exists in~\(\C{}\), and that \(F\) is a cofibration. We can certainly
form the pushout in the category of chain complexes, and by taking
cofibres (\emph{i.e.}, cokernels) of \(f\) and~\(F\) we obtain a
commutative diagram with exact rows:
\begin{equation*}
  \begin{tikzcd}
    0 \rar{} & A \rar{f} \dar{g} \po
      & B \rar{} \dar{}
      & \mathrm{coker}\,f \rar{} \dar{} & 0 \\
    & C \rar{F} & D \rar{} & \mathrm{coker}\,F \rar{} & 0
  \end{tikzcd}
\end{equation*}
By general properties of pushouts, the right-hand vertical map is an
isomorphism. Moreover, a pushout of an injective map of chain
complexes is injective again. So the diagram can be re-drawn as a map
between two short exact sequences:
\begin{equation*}
  \begin{tikzcd}
    0 \rar{} & A \rar{f} \dar{g} \po
      & B \rar{} \dar{}
      & \mathrm{coker}\,f \rar{} \dar{\cong} & 0 \\
    0 \rar{} & C \rar{F} & D \rar{} & \mathrm{coker}\,F \rar{} & 0
  \end{tikzcd}
\end{equation*}
As \(A\) and \(B\) are admissible so is \(\mathrm{coker}\,f\), by
Lemma~\ref{in_C}. Hence the isomorphic complex \(\mathrm{coker}\,F\) is
admissible. Since \(C\) is admissible by hypothesis, Lemma~\ref{in_C}
asserts that \(D = A \cup_B C\) is admissible.

Axiom Weq~1 (all isomorphisms are weak equivalences) is trivial, while
axiom Weq~2 (the gluing lemma) is known to hold in the category of
chain complexes, hence holds in particular in the present case. It is
well-known that the weak equivalences (quasi-isomorphisms) satisfy the
saturation and extension axioms (the latter being a consequence of the
five lemma, applied to long exact homology sequences).

\medbreak

(\ref{cylinder}) The mapping cylinder \(T(f)\) of a map \(f \colon C \to D\)
of chain complexes is chain homotopy equivalent to~\(D\). Thus
\(\mathbf{H}T(f) = \mathbf{H}D\), and \(T(f)\) is admissible if (and
only if) \(D\) is. Thus \(T(f)\) furnishes \(\C{}\) with a cylinder
functor satisfying the cylinder axiom since the construction is
well-known to provide such a cylinder functor on the category of all
positive chain complexes (with no admissibility assumptions).
\end{proof}

\begin{dfn}
The \emph{Hölder method \(K\)-theory of~\(R\)} is defined as the algebraic
\(K\)-theory of~\(\C\). That is, we have the \(K\)-theory
space
\begin{equation*}
  K^{\text{H}}(R) = K(\C{}) =
  \Omega \big| \mathrm{w}\mathcal{S}_{\bullet} \C{} \big|
\end{equation*}
and the \(K\)-groups
\begin{equation*}
  K^{\text{\upshape H}}_{n}(R) = K_n(\C{}) = \pi_n K(\C{}) \qquad \text{(\(n \geq 0\)).}
\end{equation*}
\end{dfn}

The group \(K^{\text{H}}_0(R)\) is the abelian group generated by
symbols \([A]\), for each object \(A \in \C{}\), subject to the
relations
\begin{itemize}
\item \([A] = [A']\) whenever there exists a weak equivalence \(A \to A'\),
\item \([B] = [A] + [B/A]\) for each cofibration \(f \colon \tikzcd A \rcof & B
  \endtikzcd\), where \(B/A = \mathrm{coker}\,(f)\).
\end{itemize}

\begin{thm}
The function \(\cH\) induces a well-defined surjective group
homomorphism \(K^{\text{\upshape H}}_0(R) \to \bR\), thus
\(K^{\text{\upshape H}}_0(R)\) is an uncountable group.
\end{thm}

\begin{proof}
Note that \(\cH(A) = \cH(A')\) whenever the complexes \(A\) and \(A'\)
are quasi-isomorphic, because the definition of~\(\cH\) only depends
on the ranks of the homology modules of the chain complexes
involved. In view of Theorems~\ref{additive} and~\ref{thm:all_R}, the statement
follows immediately from the above presentation of
\(K^{\text{\upshape H}}_0(R)\).
\end{proof}

\section{Chain complexes with finitely generated chain modules}
\label{sec:orgc847c71}
\label{sec:finite_model}

Let \(\C{}_\mathrm{f}\) denote the full subcategory
of~\(\C{}\) consisting of admissible chain complexes with
finitely generated chain modules.

\begin{thm}
\label{thm:finite_same_K}
\begin{enumerate}[{\rm (a)}]
\item The category \(\C{}_\mathrm{f}\) is a subcategory with
cofibrations and weak equivalences of the category~\(\C{}\),
and the cylinder functor on~\(\C{}\) restricts to a cylinder
functor for~\(\C{}_\mathrm{f}\).
\item The inclusion functor \(\iota \colon \C{}_\mathrm{f}
   \xrightarrow{\subset} \C{}\) induces a homotopy
equivalence
\begin{equation*}
  {}|w \mathcal{S}_{\bullet} \C{}_\mathrm{f} |
  \xrightarrow{\sim}
  {}|w \mathcal{S}_{\bullet} \C{} |
\end{equation*}
and hence an isomorphism on \(K\)-groups.
\end{enumerate}
\end{thm}

\begin{proof}
The first part is immediate. --- The second part makes use of the
Approximation Theorem \cite[, Theorem 1.6.7]{zbMATH03927168}. It is enough to
verify the following assertion:
\begin{quote}
  For any \(C \in \C{}_\mathrm{f}\), any \(D \in
  \C{}\) and any chain map \(f \colon C \to D\) there exists a
  subcomplex \(D' \subseteq D\) together with a chain map \(f' \colon C \to
  D'\) such that
  \begin{itemize}
    \item \(D'\) is an object of \(\C{}_\mathrm{f}\),
    \item the inclusion map \(i \colon D' \to D\) is a
      quasi-isomorphism,
    \item the factorisation \(f = i \circ f'\) holds.
  \end{itemize}
\end{quote}
In fact, we will inductively construct the stages of an ascending exhaustive
filtration
\begin{equation*}
  0 = F_{-1} D' \subset F_0 D' \subset F_1 D' \subset \ldots \subset
  D' = \bigcup_{k \geq 0} F_k D'
\end{equation*}
of subcomplexes of~\(D\), together with compatible chain maps
\begin{equation*}
  f'_k \colon C_{\leq k} \to F_k D'
\end{equation*}
such that
\begin{itemize}
\item each \(F_k D' \subseteq D\) consists of finitely generated modules,
\item \((F_k D')_\ell = 0\) for \(\ell > k\),
\item the inclusion \(i_k \colon F_k D' \to D\) induces
\begin{itemize}
\item a surjection \(H_k (F_k D') \to H_k(D)\),
\item an isomorphism \(H_j (F_k D') \to H_j(D)\), for all \(j < k\),
\end{itemize}
\item the composition \(i_k \circ f'_k\) coincides with the restriction
\(f|_{C_{\leq k}}\) of~\(f\) to~\(C_{\leq k}\); in particular,
\(\mathrm{im}\, (f_k) \subseteq (F_k D')_k\).
\end{itemize}

For \(k=-1\) we simply define \(F_{-1}D'\) to be the trivial chain
complex; together with the null homomorphisms \(f'_{-1}\) and
\(i_{-1}\), this choice satisfies the above conditions.

\medbreak

Assume now that, for some fixed \(k \geq -1\), we have constructed all
requisite data already.

We want to construct a suitable complex
\(F_{k+1} D'\) and a suitable chain map \(f'_{k+1} \colon C_{\leq k+1}
\to F_{k+1} D'\) such that \(i_{k+1} \circ f'_{k+1} = f|_{C_{\leq
k+1}}\). This will be done in three steps.

\medbreak

\emph{Step~1}: Let \(X = \mathrm{im}\, \big( f_{k+1} \colon C_{k+1} \to
D_{k+1} \big)\), and let \(E\) denote the chain complex
\begin{equation*}
  E \colon \quad
  \ldots \to 0 \to X \xrightarrow{\partial} (F_k D')_k
  \to (F_k D')_{k-1} \to (F_k D')_{k-2} \to \ldots
\end{equation*}
which "extends" the complex \(F_k D'\) into chain level~\(k+1\). The
map "\(\partial\)" is the restriction of the boundary map \(d^D_{k+1}
\colon D_{k+1} \to D_k\) of~\(D\). This makes sense since
\begin{equation*}
  d^D_{k+1} (X) = \mathrm{im}\, (d^D_{k+1} \circ f_{k+1})
  = \mathrm{im}\, (f_k \circ d^C_{k+1})
  \subseteq \mathrm{im}\, (f_k) \subseteq (F_k D')_k
\end{equation*}
so \(\partial\) is indeed well-defined. This complex~\(E\) is a
subcomplex of~\(D\), consists of finitely generated \(R\)-modules,
and comes with a map \(C_{\leq k+1} \to E\) which coincides with
\(f|_{C_{\leq k+1}}\) up to the inclusion map \(E \to D\). Also, the
inclusion induces a surjective map \(H_k (E) \to H_k(D)\), because in
the commutative diagram
\begin{equation*}
\begin{tikzcd}
  Z_k(E) \rar{=} \dar{} & Z_k(F_k D') \rar{=} & H_k(F_k D') \dar{} \\
  H_k(E) \ar[rr]{} && H_k(D) \\
\end{tikzcd}
\end{equation*}
the right-hand vertical arrow is surjective by induction hypothesis,
hence so is the composition of the left-hand and bottom arrows, hence
so is the bottom arrow by itself.

\smallbreak

\emph{Step~2}: Let \(Y\) be a finitely generated submodule of the module of
cycles \(Z_{k+1} D = \ker (d^D_{k+1} \colon D_{k+1} \to D_k)\) mapping
onto the homology module \(H_{k+1}(D) = Z_{k+1} D / B_{k+1} D\); such a
\(Y\) exists since \(H_{k+1}(D)\) is finitely generated. We let \(E'\)
denote the chain complex
\begin{equation*}
  E' \colon \quad
  \ldots \to 0 \to X+Y \xrightarrow{\partial'} (F_k D')_k
  \to (F_k D')_{k-1} \to (F_k D')_{k-2} \to \ldots
\end{equation*}
which "extends" the complex~\(E\). The map \(\partial'\) is defined to
be the map \(\partial\) from~\(E\) on~\(X\), and the zero homomorphism
on~\(Y\); in other words, \(\partial'\) is the restriction of
\(d^D_{k+1}\) to the submodule \(X+Y\) of~\(D_{k+1}\). Since \(B_k E' =
B_k E\) and \(Z_k E' = Z_k E\), the inclusion gives a surjection
\begin{equation}
\label{map:rho}
\rho \colon H_{k} E' = H_{k}  E \to H_{k} D \ .
\end{equation}
The composition of the map \(f'_k\) with the inclusions \(E \subseteq
E' \subseteq D\) coincides with~\(f|_{C_{\leq k+1}}\), and \(E'\)
consists of finitely generated \(R\)-modules. Since \(Y \subseteq
Z_{k+1} E'\), the inclusion \(E' \subseteq D\) induces a surjection
\(H_{k+1} E' = Z_{k+1} E' \to H_{k+1} D\). Note that
\(H_{k+1} E'\) is a submodule of~\(X+Y\), hence is finitely
generated.

\smallbreak

\emph{Step~3}: Let \(K = \ker(\rho)\) be the kernel of the
homomorphism~\(\rho\) of~(\ref{map:rho}), which is a finitely generated
\(R\)-module.  Since \(H_k E'\) is a quotient of~\(Z_k E'\) we can
find a finitely generated submodule \(\tilde Z\) of~\(Z_k E'\) mapping
surjectively to~\(K\) under the canonical quotient map \(Z_k E' \to
H_k E'\). Since \(E' \subseteq D\) we know \(\tilde Z \subseteq Z_k
D\); but \(\tilde Z\) maps onto~\(K\), hence maps to~0 in~\(H_k D\),
thus \(\tilde Z \subseteq B_k D\). Thus we can choose a finitely
generated submodule~\(Z\) of~\(D_{k+1}\) such that \(d^D_{k+1}(Z) =
\tilde Z\). We let \(F_{k+1} D'\) denote the chain complex
\begin{equation*}
  F_{k+1} D' \colon \quad
  \ldots \to 0 \to X+Y+Z \xrightarrow{\partial''} (F_k D')_k
  \to (F_k D')_{k-1} \to \ldots
\end{equation*}
which "extends" the complex~\(E'\). The map \(\partial''\) is the
restriction of~\(d^D_{k+1}\); this makes sense as \(d^D_{k+1}(Z) =
\tilde Z \subseteq Z_k E' \subseteq (F_k D')_k\). By construction,
\(F_{k+1} D'\) is a subcomplex of~\(D\) consisting of finitely
generated \(R\)-modules. It comes with a map \(f'_{k+1} \colon C_{\leq
k+1} \to F_{k+1} D'\) whose composition with the inclusion \(F_{k+1}D'
\subseteq D\) coincides with \(f|_{C_{\leq k+1}}\); indeed,
\(f'_{k+1}\) is just the co-restriction of \(f|_{C_{\leq
k+1}}\). Since \(Y \subseteq Z_{k+1} F_{k+1}D' = H_{k+1}(F_{k+1}D')\),
the inclusion gives a surjective map \(H_{k+1}(F_{k+1}D') \to
H_{k+1} D\).

It remains to check that the map induced on~\(H_k\) is an
isomorphism. To this end, recall first that the map~\(\rho\)
from~\ref{map:rho} above is surjective with kernel~\(K\). On the other hand,
the canonical map
\begin{equation*}
  \rho' \colon H_k E' \to H_k(F_{k+1} D')
\end{equation*}
is onto since \(Z_k E' = Z_k F_{k+1} D'\), and since
\begin{equation*}
  B_k E' = d^D_{k+1} (X+Y) \,\subseteq\,
  d^D_{k+1} (X+Y+Z) = B_k F_{k+1} D' \ .
\end{equation*}
The kernel of~\(\rho'\) is the image of~\(Z\) in~\(H_k E'\) which, by
choice of~\(Z\), equals the image of~\(\tilde Z\) in~\(H_k(E')\)
which, by choice of~\(\tilde Z\), equals~\(K\). Thus we have a
commutative diagram of exact sequences
\begin{equation*}
\begin{tikzcd}
  K \dar{=} \rar{} & H_k E' \rar{\rho'} \dar{=} & H_k (F_{k+1} D')
  \rar{} \dar{H_k(i_{k+1})} & 0 \\
  K \rar{} & H_k E' \rar{\rho} & H_k D \rar{} & 0 \\
\end{tikzcd}
\end{equation*}
which implies that the right-hand vertical map, which is induced by
the inclusion \(i_{k+1} \colon F_{k+1}D' \to D\), is an isomorphism.

\medbreak

To finish the proof, we define \(D' = \bigcup_{k \geq 0} F_k D'\), and
define \(f'\) and~\(i\) by saying that they agree with \(f'_{k}\)
and~\(i_k\), respectively, on~\(F_k D'\). The triple \((D', f', i)\)
satisfies all the requirements.
\end{proof}

\section{Computing the infinite Euler characteristic directly from the chain modules}
\label{sec:org158824d}
\label{sec:compute}

For \emph{bounded} chain complexes of finitely generated modules it is
well-known that the (usual) Euler characteristic, defined as the
alternating sum of the ranks of the homology modules, can equally be
computed as the alternating sum of the ranks of the chain modules. We
discuss the analogous statement for~\(\cH\) in this section.

\smallbreak

Let \(C\) be a pre-admissible chain complex of finitely generated
\(R\)-modules; that is, we insist that \(C\) has finitely generated
chain modules and not just finitely generated \emph{homology} modules. In
addition to the sequence~\(\mathbf{H}C\) we define the further
sequences
\begin{gather*}
  \mathbf{B}C = \big((-1)^{n-1} \mathrm{rank}\,B_{n-1}(C) \big)_{n
  \in \mathbb{N}} \\
  \intertext{and}
  \mathbf{R}C = \big( (-1)^{n-1} \mathrm{rank}\,C_{n-1} \big)_{n
  \in \mathbb{N}}
\end{gather*}
where \(B_k(C) = \mathrm{im}\, \big( \partial: C_{k+1} \to C_k \big)\)
is the module of \(k\)-boundaries of~\(C\).

\begin{thm}
Suppose that \(C\) is a positive chain complex of finitely generated
\(R\)-modules such that \(\mathbf{B}C = \mathrm{H} \hbox{-}
o(n)\). The sequence \(\mathbf{H}C\) is H-convergent (\emph{i.e.}, \(C\) is
pre-admissible) if and only if \(\mathbf{R}C\) is H-convergent, in
which case \(\cH(C) = \mathrm{\relax H}\hbox{-}\lim \mathbf{H}C =
\mathrm{\relax H}\hbox{-}\lim \mathbf{R}C\).
\end{thm}

\begin{proof}
Fix \(n \in \mathbb{N}\) for the moment. We want to compare the
\(n\)th terms of the sequences \(M\mathbf{H}C\) and
\(M\mathbf{R}C\). To this end, we note that the (classical) Euler
characteristic of the truncated chain complex \(C_{<n}\), which is a
bounded complex, is a multiple of the \(n\)th term of the sequence \(M
\mathbf{H} C\):
\begin{equation*}
  \chi(C_{<n})
  = \sum_{k=0}^{n-1} (-1)^k \mathrm{rank}\, H_k(C_{<n})
  = \sum_{k=0}^{n-1} (-1)^k \mathrm{rank}\, C_k
  = n \cdot (M\mathbf{R}C)_n \ ,
\end{equation*}
Now \(H_k(C_{<n}) = H_k C\) for \(k<n-1\), while \(H_{n-1} (C_{<n}) =
\ker (C_{n-1} \to C_{n-2}) = Z_{n-1}C\) is the \((n-1)\)st cycle module
of~\(C\). The short exact sequence
\begin{equation*}
  0 \to B_{n-1} C \to Z_{n-1} C \to H_{n-1} C \to 0
\end{equation*}
shows that
\begin{equation*}
  \mathrm{rank}\, Z_{n-1}C = \mathrm{rank}\, H_{n-1} C +
  \mathrm{rank}\, B_{n-1} C \ .
\end{equation*}
Putting all these together, we get
\begin{align*}
  (M\mathbf{R}C)_n &= \frac1n \sum_{k=0}^{n-1} (-1)^k \mathrm{rank}\, H_k
  (C_{<n}) \\
  & = \frac1n \sum_{k=0}^{n-1} (-1)^k \mathrm{rank}\, H_k(C)
  + \frac{(-1)^{n-1}}{n} \mathrm{rank}\, B_{n-1}(C) \\
  & = (M\mathbf{H}C)_n + \frac1n (\mathbf{B}C)_n \ .
\end{align*}
Consequently, for \(k > 0\)
\begin{equation*}
  M^k \mathbf{R}C = M^k \mathbf{H}C + M^{k-1} (\mathbf{B}C/n) \ .
\end{equation*}
The hypotheses say that for \(k \gg 0\), the sequence \(M^{k-1}
(\mathbf{B}C/n)\) converges to~\(0\). This shows that \(M^k
\mathbf{R}C\) converges if and only if \(M^k \mathbf{H}C\) converges,
and that the limits agree in that case.
\end{proof}

\section*{Concluding remarks}
\label{sec:org0ef6892}
\label{sec:remarks}

The specifics of the Hölder summation method are crucial for the
present treatment. From the perspective of summability theory,
\(M^k\)-convergence is the same as convergence with respect to the
Cesàro summation method \((C,k)\), see \cite[p.448]{zbMATH02621104} or
\cite[Theorem 3.1.16]{zbMATH01544061}.  The method \((C,k)\) has the
advantage of not being defined iteratively, but it seems difficult to
adapt the current proof of additivity of~\(\cH\). Nevertheless, to
obtain larger classes of admissible chain complexes, the use of other
(stronger) summation methods can be considered. A variant using the
Abel summation method will be the topic of a forthcoming paper.

We have shown that the group~\(K^{\text{\upshape H}}_0(R)\) is
uncountable. It would be interesting to have a precise identification
of the group, to investigate possible ring structures it supports, and
to find natural topologies or even geometric structures
on~\(K^{\text{\upshape H}}_0(R)\).

\subsubsection*{ }
\label{sec:orgfbb036a}
\bibliographystyle{alpha}
\bibliography{/home/huette/SYNC-notes/Maths/Euler_characteristics/chi1}

\begin{thebibliography}{Gub04}

\bibitem[BL08]{zbMATH05229957}
Clemens Berger and Tom Leinster.
\newblock The {Euler} characteristic of a category as the sum of a divergent
  series.
\newblock {\em Homology Homotopy Appl.}, 10(1):41--51, 2008.

\bibitem[Boo00]{zbMATH01544061}
Johann Boos.
\newblock {\em Classical and modern methods in summability}.
\newblock Oxford Math. Monogr. New York, NY: Oxford University Press, 2000.

\bibitem[Die57]{zbMATH03127712}
P.~Dienes.
\newblock The {Taylor} series. {An} introduction to the theory of functions of
  a complex variable.
\newblock New {York}: {Dover} {Publications} {Inc}. {X}, 552 p. (1957)., 1957.

\bibitem[Gub04]{zbMATH02105767}
Joseph Gubeladze.
\newblock Toric varieties with huge {Grothendieck} group.
\newblock {\em Adv. Math.}, 186(1):117--124, 2004.

\bibitem[H{\"o}l82]{zbMATH02705122}
O.~H{\"o}lder.
\newblock {G}renzwerthe von {R}eihen an der {C}onvergenzgrenze.
\newblock {\em Math. Ann.}, 20:535--549, 1882.

\bibitem[Kui65]{zbMATH03211541}
N.~H. Kuiper.
\newblock The homotopy type of the unitary group of {Hilbert} space.
\newblock {\em Topology}, 3:19--30, 1965.

\bibitem[Min87]{zbMATH03998548}
J.~A. Mingo.
\newblock K-theory and multipliers of stable {{\({C}^ *\)}}-algebras.
\newblock {\em Trans. Am. Math. Soc.}, 299:397--411, 1987.

\bibitem[Sch13]{zbMATH02621104}
I.~Schur.
\newblock {\"U}ber die {\"a}quivalenz der \emph{Ces{\`a}ro}schen und
  \emph{H{\"o}lder}schen {Mittelwerte}.
\newblock {\em Math. Ann.}, 74:447--458, 1913.

\bibitem[Wal85]{zbMATH03927168}
Friedhelm Waldhausen.
\newblock Algebraic {K}-theory of spaces.
\newblock Algebraic and geometric topology, {Proc}. {Conf}., {New}
  {Brunswick}/{USA} 1983, {Lect}. {Notes} {Math}. 1126, 318-419 (1985)., 1985.

\bibitem[WO93]{zbMATH00412105}
Niels~Erik Wegge-Olsen.
\newblock {\em {{\(K\)}}-theory and {{\(C^*\)}}-algebras: a friendly approach}.
\newblock Oxford: Oxford University Press, 1993.

\bibitem[Zel51]{zbMATH03070559}
Karl Zeller.
\newblock Allgemeine {Eigenschaften} von {Limitierungsverfahren}.
\newblock {\em Math. Z.}, 53:463--487, 1951.

\end{thebibliography}
\end{document}